\documentclass[a4paper,11pt]{article}
\usepackage{graphicx}
\usepackage{amsfonts}
\usepackage{amsmath}
\usepackage[utf8]{inputenc}
\usepackage{url}
\usepackage{color}
\usepackage{amsthm}
\usepackage{amssymb}
\usepackage{enumerate}

\newcommand{\M}{\color[rgb]{.46,.2,.1}}

\newtheorem {Proposition}{Proposition}[section]
\newtheorem {Lemma}[Proposition] {Lemma}
\newtheorem {Theorem}[Proposition]{Theorem}
\newtheorem {Corollary}[Proposition]{Corollary}

\def\N{\mathbb{N}}

\def\R{\mathbb{R}}

\numberwithin{equation}{section}

\title{A note on the Regularity of \\ Center-Outward Distribution and Quantile Functions}

\author{Eustasio del Barrio$^{(1)}$\footnote{This author has been partially supported by FEDER, Spanish Ministerio de Economía y Competitividad, GrantMTM2017-86061-C2-1-P and Junta de Castilla y León, Grants VA005P17 and VA002G18}, Alberto Gonz\'alez-Sanz,$^{(2)}$ and Marc Hallin$^{(3)}$\\  $\,$ \\ 
{ $^{(1)(2)}$IMUVA, Universidad de Valladolid, Spain} \\ 
$^{(3)}$ECARES and D\' epartement de Math\' ematique, Universit\' e libre de Bruxelles, Belgium\\ $\,$ \\ 
$^{(1)}$tasio@eio.uva.es \quad $^{(2)}$alberto.gonzalez.sanz.96@gmail.com\quad $^{(3)}$mhallin@ulb.ac.be}

\begin{document}

\maketitle

\begin{abstract}
We provide sufficient conditions under wich the center-outward distribution and quantile functions introduced in Chernozhukov et al.~(2017) and Hallin~(2017) are homeomorphisms, thereby extending a recent result by Figalli \cite{Fi2}. Our approach
relies on Cafarelli's  classical regularity theory  for the solutions of the Monge-Amp\`ere equation, but has to deal with  difficulties related with the unboundedness at the origin of the density of the spherical uniform  reference measure. Our conditions are satisfied by probabillities on Euclidean space with a general (bounded or unbounded) convex support which are not covered in~\cite{Fi2}.
We provide some additional results about center-outward distribution and quantile functions, including the fact that quantile sets exhibit some weak form of convexity.
\end{abstract}

\noindent \textit{Keywords:} Optimal transportation; Monge-Amp\`ere equation; multivariate ranks; quantile contours.

\section{Introduction: center-outward distribution and quantile functions}

Univariate distribution and quantiles functions, together with their empirical counterparts and the closely related concepts of ranks and order statistics, count among the most
fundamental and useful tools in mathematical statistics. Ranks indeed are not just distribution-free: in models driven by noise with unspecified density, they generate the sub-$\sigma$-field of all distribution-free events (see \cite{delBarrioetal}), which is also the largest  sub-$\sigma$-field independent, irrespective of the underlying distribution, of the minimal sufficient $\sigma$-field generated by the order statistic; suitable rank-based procedures achieve optimality in several senses in nonparametric testing as well semiparametric efficiency (see, e.g.~\cite{HajekSidakSen}, \cite{H94}, \cite{HTribel00}, \cite{HallinWerker}). A major limitation of  the classical concepts of ranks and quantiles, however,  is that, due to the absence of a canonical ordering of $\mathbb{R}^d$ for $d\geq 2$, they do not readily extend to the multivariate context. 

The problem  is not new, and numerous attempts have been made to fill that gap by defining multivariate versions of distribution and quantiles  functions, with the ultimate goal of constructing   suitable mutivariate versions of classical rank- and quantile-based  inference procedures. The traditional definition of a multivariate distribution function is somewhat helpless in that respect, and does not   produce any satisfactory concept of quantiles---let alone a satisfactory concept of ranks (see \cite{Genest}). The componentwise approach, closely related with   copula transforms, has been studied intensively (see~\cite{PuriSen}), but does not even enjoy distribution-freeness. Nor do the so-called {\it spatial ranks} (\cite{Oja99}, \cite{Oja2010}) inspired by the  L$_1$ characterization of univariate quantiles. The whole theory of statistical depth (see \cite{Bob1},  \cite{Bob2} for authoritative surveys), in a sense, is motivated by the same objective of providing a (data-driven) ordering of $\mathbb{R}^d$ and adequate concepts of multivariate ranks (\cite{He06}) and quantiles (\cite{He97}); here again, the resulting notions fail to be distribution-free. As for the {\it Mahalanobis ranks and signs} considered, e.g.  in \cite{HPd02a}, \cite{HPd02b} or \cite{HPdV10}, they do enjoy distribution-freeness and all the desired properties expected from ranks---under the restrictive assumption, however, of elliptical symmetry. 

 This shortcoming of all available solutions has motivated the introduction, in \cite{Chernouzhukovetal} and \cite{Hallin},  of the measure transportation-based concepts of {\it Monge-Kantotovich depth},  {\it center-outward distribution} and {\it quantile functions, ranks,} and {\it signs}. These center-outward  concepts, unlike all previous ones,  are shown (see \cite{Hallin}, \cite{delBarrioetal}) to  enjoy all the properties  that make their univariate counterpart a fundamental and successful tool for statistical inference; we refer to~\cite{delBarrioetal} for more references and further discussion. 

%The analysis in \cite{HallinWerker} gives some interesting insights into the conditions that a successful (in the sense of semiparametric efficiency) extension of multivariate ranks should enjoy. There are two key aspects of univariate ranks which ensure the semiparametric efficiency: distribution freeness and maximal invariance (with respect to the class of smooth monotone transformations). 

Let ${\rm P}$ be a Borel probability measure on the real line with finite second moment and continuous distribution function $F$ and denote by ${\rm U}_{[0,1]}$  the uniform distribution on $(0,1)$: then $F$  is a solution to {\it Monge's quadratic  transportation problem}, that is, 
$$\int_{\mathbb{R}} |x-F(x)|^2 d{\rm P}(x)=\min_{T:\, T\sharp {\rm P}={\rm U}_{[0,1]}} \int_{\mathbb{R}} |x-T(x)|^2 d{\rm P}(x)$$
(see, e.g.,  \cite{Vi}) where $T\sharp P$ denotes the   {\it push  forward of   ${\rm P}$ by $T$}---namely, the distribution    of $T(X)$   under~$X\sim{\rm P}$ ($T$ a measurable map from $\mathbb{R}$ to~$(0,1)$). With generalization to higher dimension in mind, however, \cite{Chernouzhukovetal}  and~\cite{Hallin} rather  consider $F_\pm(x)=2F(x)-1$, the so-called {\it center-outward distribution function} of $\rm P$, satisfying  the transportation problem 
$$\int_{\mathbb{R}} |x-{F}_\pm(x)|^2 d{\rm P}(x)=\min_{T:\, T\sharp {\rm P}={\rm U}_{1}} \int_{\mathbb{R}} |x-T(x)|^2 d{\rm P}(x),$$
where $\mathrm{U}_1$ is  the uniform distribution over $(-1,1)$, the one-dimensional unit ball $\mathbb{B}_1$. Clearly, $F_\pm$ and~$F$ carry the same information about $\rm P$. 

The latter definition, indeed,  readily extends to arbitrary dimensions. Let  ${\rm P}$ denote a Borel probability measure on $\mathbb{R}^d$ with finite second-order  moments and Lebesgue density $p$.  Measure transportation theory  
 (see, e.g., Theorem 2.12 in \cite{Vi})  tells us that 
there exists a ${\rm P}$-a.s. unique map~${\mathbf F}_\pm$ such that 
\begin{equation}\label{Mongepb}
\int_{\mathbb{R}^d} |\mathbf{x}-{\mathbf F}_\pm(\mathbf{x})|^2 d{\rm P}(\mathbf{x})=\min_{{\mathbf T}:\, {\mathbf T}\sharp {\rm P}=\mathrm{U}_d} \int_{\mathbb{R}^d} |\mathbf{x}-{\mathbf T}(\mathbf{x})|^2 d{\rm P}(\mathbf{x})
\end{equation}
where $|\mathbf{x}|$ stands for the Euclidean norm of $\mathbf{x}$ and  
 $\mathrm{U}_d$ denotes the \textit{uniform}
distribution over the open~$d$-dimensional unit ball $\mathbb{B}_d$. The center-outward distribution function is defined as a  solution~${\mathbf F}_\pm$   of this optimal transportation problem.

By uniform over $\mathbb{B}_d$ we refer to %Uniformity of $\mathrm{U}_d$ here refers to
  {\it spherical} uniformity, that is, $\mathrm{U}_d$ here corresponds to the uniform choice of a direction on the unit sphere $\mathbb{S}_{d-1}:=\bar{\mathbb{B}}_d -\mathbb{B}_d$ in $\mathbb{R}^d$ combined with an independent uniform choice in~$(0,1)$ of a distance to the origin. A simple change of variable shows that $\mathrm{U}_d$ has density 
\begin{equation}\label{densityUd}
 u_d(\mathbf{x})=\frac 1 {a_d|\mathbf{x}|^{d-1}} I\big[\mathbf{x}\in \mathbb{B}_d-\{\mathbf{0}\}\big]
 \end{equation}
 where $a_d=2\pi^{d/2}/\Gamma(d/2)$ denotes the area (the $(d-1)$-dimensional Hausdorff measure, see, e.g.,~\cite{EvansGariepy}) of the sphere $\mathbb{S}_{d-1}$. Note the singularity at the origin since~$\mathbf{x}\mapsto 1/|\mathbf{x}|^{d-1}$ is infinite at ${\mathbf{x}}={\mathbf{0}}$; while we safely can neglect  $\mathbf{0}$ itself, which has measure zero, by putting $u_d(\mathbf{0})=0$, $ u_d$ nevertheless remains unbounded in the vicinity of $\mathbf{0}$.

This definition of the center-outward distribution function as the solution of a  quadratic transportation problem suffers from two major limitations. First, %unlike its univariate counterpart,
  finite second-order moments are needed to make 
sense of the underlying optimization problem \eqref{Mongepb}.   Second, the distribution function~${\mathbf F}_\pm$ based on \eqref{Mongepb}  is only defined $\rm P$-a.s.; this means, for instance,  that ${\mathbf F}_\pm$
is not well defined outside the support of $\rm P$. 

The first of these two limitations has been  relaxed in \cite{Hallin} thanks to a celebrated theorem by  McCann~\cite{McCann}.  Under the assumption that $\rm P$  has finite second-order moments,  Brenier in 1991 had shown that   optimal transportation maps (hence, all versions of the $\rm P$-a.s.\ unique   solution ${\mathbf F}_\pm$ of Monge's problem~\eqref{Mongepb})  coincide $\rm P$-a.s.\  with  the Lebesgue-a.e.~gradient $\nabla \varphi$ of a convex function\footnote{The notation $\nabla \varphi$ here is used for the Lebesgue-a.e.~gradient of $\varphi$, that is, $\nabla \varphi({\bf x})$ is defined as the gradient at $\bf x$ of~$\varphi$ whenever $\varphi$ is differentiable at $\bf x$---which, for a convex $\varphi$, holds Lebesgue-a.e. Note  that, contrary to~$\nabla \varphi$, which is a.e.~unique, $\varphi$ is not---unless we impose, without loss of generality (see, e.g., Lemma 2.1 in \cite{delBarrioLoubes}), that $\varphi({\bf 0})=0$.}~$\varphi$, which has the interpretation of a {\it potential}.  More precisely,~${\mathbf F}_\pm$ a.s.\  is of the form~$\nabla \varphi$ where~$\varphi$  {\it (i)} is lower semicontinous (lsc in the sequel),   {\it (ii)} is convex, and {\it (iii)} is such that~$\nabla \varphi\sharp {\rm P}=\mathrm{U}_d$.  McCann~\cite{McCann} further showed that these last three conditions uniquely determine $\nabla \varphi$, even in the absence of  second moment assumptions, while under finite second-order moments, $\nabla \varphi$ is a solution of Monge's problem~\eqref{Mongepb}. Thus,  putting 
\begin{equation}\label{defF}{\mathbf F}_\pm(\mathbf{x}) :=\nabla \varphi(\mathbf{x}) \quad \mbox{$\mathbf{x}$-a.e.~in } \mathbb{R}^d ,
\end{equation}
 the center-outward distribution function ${\mathbf F}_\pm$ is no longer characterized as the almost surely unique  solution of an  optimization problem  \eqref{Mongepb} requiring finite moments of order two but as the unique a.e.~gradient $\nabla \varphi$ of a convex function pushing $\rm P$ forward to $\mathrm{U}_d$.  {\M We nevertheless conform to the common usage of improperly calling $\nabla\varphi$ the optimal transport pushing $\rm P$  forward  to ${\rm U}_d$.}

While taking care of the moment assumption---existence of second-order moments indeed is an embarrassing assumption when distribution and quantile functions are to be defined---the second limitation still remains. The non-unicity of ${\mathbf F}_\pm :=\nabla \varphi$, however, disappears if $\rm P$ is such that $\varphi$ is everywhere differentiable. That this is indeed the case was shown  by Figalli  in 2018 \cite{Fi2} for~$\rm P$ in the so-called class of distributions {\it with nonvanishing densities}\footnote{Precisely, the distributions $\rm P$ with densities $p$ and support ${\mathcal X}=\mathbb{R}^d$ satisfying Assumption~A below.}. For any~$\rm P$ in that class of distributions, Figalli actually establishes that  $\nabla \varphi ({\bf x})$ is a gradient for all $\bf x$ and, when restricted to 
$$\mathbb{R}^d_{({\bf 0})}:=\mathbb{R}^d\setminus\{{\bf x}: \nabla \varphi ({\bf x})\neq{\bf 0}\},$$
 a  homeomorphism between $\mathbb{R}^d_{({\bf 0})}$ and the punctured ball $\mathbb{B}_d\setminus\{{\bf 0}\}$. The latter property is quite essential if sensible---namely, closed, continuous, connected, and nested---quantile regions and contours, based on an inverse\footnote{See Section~2.1 for a precise definition.} ${\mathbf Q}_\pm$ of ${\mathbf F}_\pm$, are to be defined: see \cite{Hallin} and \cite{delBarrioetal}.

 The goal of this paper is to provide simple sufficient conditions for   Figalli's results to hold beyond the assumption of nonvanishing densities; we more particularly consider  distributions with (bounded or unbounded) convex supports. Beyond other theoretical considerations, these are the key properties required to prove a.s. convergence of the empirical center-outward distribution functions to their theoretical counterparts (see \cite{delBarrioetal}). Hence, the results of the present paper also are  extending the validity of the center-outward Glivenko-Cantelli theorem in that reference.

% $\mathbf{F}_\pm$ and the corresponding quantile function  $\mathbf{Q}_\pm$ are well defined for every in the interior of the support of $P$ and every point in the open unit ball, respectively. We note that we can define $\mathbf{F}_\pm=\partial \varphi$ and 
%$\mathbf{Q}_\pm=\partial \psi$ as multivalued maps. Hence, the main issue here is to guarantee that $\mathbf{F}_\pm$ and $\mathbf{Q}_\pm$ 
%are single valued maps. As an automatic consequence one obtains that $\mathbf{F}_\pm$ and $\mathbf{Q}_\pm$ are continuous functions, since gradients of convex functions are continuous in the set of differentiability of the corresponding potentials (see, e.g., Corollary 25.5.1 in \cite{Rockafellar}).
%An additional consequence of our results will be that $\mathbf{F}_\pm$ and $\mathbf{Q}_\pm$ are homeomorphisms, inverses of each other (after the removal of some small troubling sets). Beyond other theoretical considerations, these are the key properties required to prove a.s. convergence of the empirical center-outward distribution functions to their theoretical counterparts (see \cite{delBarrioetal}). Hence, the results here extend the validity of the Glivenko-Cantelli results in the cited reference.

From a technical point of view, our main result   is Theorem \ref{th:homeomorphism} below, which relies on the classical regularity theory for solutions of  Monge-Amp\`ere equations associated with the name of of Caffarelli  (see \cite{Ca1,Ca2,Ca3}), as  discussed in Section \ref{Section2}. The use of that theory
to investigate the regularity of optimal transportation maps between two probabilities typically requires that both probabilities have  densities that are bounded and bounded away from zero over their respective supports.  Recently, under a local version of this condition, a very general regularity result of this kind has been given in~\cite{Fi3}. However, the spherical uniform reference measure  $\mathrm{U}_d$ considered here, in  dimension $d\geq 1$, yields unbounded densities at the origin, so that the results in \cite{Fi3} do not apply.\footnote{Note that the choice of the spherical  uniform reference is not a whimsical one. It preserves the independence between  $\Vert{\mathbf F}_\pm\Vert$ and ${\mathbf F}_\pm/\Vert{\mathbf F}_\pm\Vert$ (extending the independence, for $d=1$, between $\vert{F}_\pm\vert$ and sign$({F}_\pm)$) and produces simple and easily interpretable quantile contours with prescribed probability content (we refer to \cite{Hallin} for details). }
To our knowledge, the only reference dealing with this kind of unbounded density is \cite{Fi2} which, however,  requires  $\rm P$ to be supported on the whole space $\mathbb{R}^d$. Here we extend the result in \cite{Fi2} to cover the case of $\rm P$ with (bounded or unbounded) convex  supports.

The sequel of this paper is organized as follows. Our main regularity result  is established in Section~2, along with a succint account of the main elements of Cafarelli's   theory and some auxiliary results. We conclude with Section 3, which presents some new results on center-outward distribution and quantile functions. These include an asymptotic invariance property 
%of center-outward distribution functions (
 extending a well-known feature of classical univariate distribution functions and  the ability of   quantile contours to  capture the shape of the bounded support of a  probability measure by convergencing (in Hausdorff distance) to the boundary of the support. Finally, we include a result on the geometry of quantile sets, showing that they turn out to exhibit a limiting form of ``lighthouse convexity''.

\section{Regularity of center-outward distribution  and quantile functions}\label{Section2}
\subsection{Center-outward quantile functions}\label{Section21}
The Introduction was focused on the distribution functions ${\bf F}_\pm$. Exchanging the roles of $\rm P$ and~${\rm U}_d$, we could have emphasized transportation from the unit ball to the support of $\rm P$, leading to the definition of the center-outward quantile function ${\bf Q}_\pm$ with, {\it mutatis mutandis}, the same   comments. 
 
% The objective of this paper is to show that Figalli's result actually extends to broader classes of distributions among which, under adequate regtlarity conditions, the compactly supported ones. The paper is organized as follows.  \textcolor{red}{later ... }

%
% the contributions of this paper is a relaxation, for a broad class of distributions $\rm P$, of the second limitation. Unicity of  ${\mathbf F}_\pm :=\nabla \varphi$ indeed follows if  $\varphi$ is everywhere differentiable. A first result on the differentiability of  $\varphi$ was established by Figalli  in 2018 \cite{Fi2} for $\rm P$ in the so-called class of distributions {\it with nonvanishing densities}. 
%
%
%
%Similarly, exchanging the roles of $\rm P$ and $\mathrm{U}_d$, it follows from McCann's Theorem that there exists a unique $\nabla \psi $ with~$\psi$ {\it (i)} lower semicontinous,   {\it (ii)} convex, and {\it (iii')}~such that $\nabla \psi\sharp \mathrm{U}_d= {\rm P}$. Call $\mathbf{Q}_\pm:=\nabla \psi$   the {\it center-outward quantile function} or, simply   {\it quantile function} of $\rm P$. 
%
%Although  ${\mathbf F}_\pm$ and ${\mathbf Q}_\pm$ no longer are defined as solutions of an optimization problem, we nevertheless   stick to the common use of calling them   optimal transportation maps, from $\rm P$ to $\mathrm{U}_d$ and from~$\mathrm{U}_d$ to $\rm P$, respectively.  
%
%
%
%\bigskip
%*********************************************************************
%
%\bigskip
%
%\bigskip
%Throughout this paper we will use the following conventions.
Let $\rm P$ denote a Borel probability measure over $\mathbb{R}^d$ with Lebesgue density $p$.  While the center-outward distribution function is defined  as the optimal  (in the McCann sense) transport pushing~$\rm P$ forward to ${\rm U}_d$, the  center-outward quantile map or  quantile function $\mathbf{Q}_\pm$ of $\rm P$ 
 is defined as  the optimal transport pushing ${\rm U}_d$ forward to $\rm P$. Namely,  
\begin{equation}\label{defQ}
\mathbf{Q}_\pm(\mathbf{u}):=\nabla \psi(\mathbf{u})\quad \mbox{$\mathbf{u}$-a.e.~in } \mathbb{B}_d
\end{equation}
where $\nabla \psi$ is, in agreement with McCann's Theorem, the unique a.e.\ gradient of a convex function~$\psi$ with domain containing $\mathbb{B}_d$\footnote{We adhere to the usual convention of considering that a function defined on $A\subset \mathbb{R}^d$ is convex if it can be extended to a convex function on $\mathbb{R}^d$ with  values in  $\mathbb{R}\cup\{\infty\}$; the domain of the convex function is then redefined as  the set where it takes finite values.} such that~$\nabla \psi\sharp \mathrm{U}_d~\!=~\!{\rm P}$. Again, imposing, without loss of generality,\footnote{Indeed, two convex functions with a.e.~equal gradients on an open convex set are equal up to an additive constant 
(see, e.g., Lemma 2.1 in \cite{delBarrioLoubes}).} that~$\psi({\bf 0})=0$, the convex potential $\psi$ is uniquely defined  
%
%McCann's theorem and the fact that  imply that the convex potential $\psi$ is uniquely defined in $B_1$ by \eqref{defQ} plus the condition
%$\psi(\mathbf{0})=\mathbf{0}$. $\psi$ is finite in $B_1$
 and a.e.~differentiable over $\mathbb{B}_d$. 
% $\mathbf{Q}_\pm$ is defined through
%\eqref{defQ} at every differentiability point of $\psi$.
 We extend $\psi$ to a lsc convex function on  $\mathbb{R}^d$ with the standard procedure
of setting $\psi(\mathbf{u}):=\liminf_{\mathbf{z}\to\mathbf{u},|\psi(\mathbf{z})|<1} \psi(\mathbf{z})$ if~$|\mathbf{u}|=1$ and~$\psi(\mathbf{u}):=+\infty$ for $\mathbf{u}\notin \bar{\mathbb{B}}_d$  (see, e.g. (A.18) in \cite{Fi}). With this extension,  $\varphi$ is  the Legendre transform of $\psi$, that is, 
\begin{equation}\label{defPhi}
\varphi(\mathbf{x})=\psi^*(\mathbf{x}):=\sup_{\mathbf{u}\in \mathbb{B}_d} (\langle\mathbf{u},\mathbf{x} \rangle-\psi(\mathbf{u})),\quad \mathbf{x}\in \mathbb{R}^d.
\end{equation}
We observe that the domain of $\varphi$ is $\mathbb{R}^d$ and  that $\varphi$, being the sup of a 1-Lipschitz function, is \linebreak also~1~\!-~\!Lipschitz. In particular, 
%we have
%that,
 for almost every $\mathbf{x}\in\mathbb{R}^d$, $\varphi$ is differentiable with~$|\nabla \varphi(\mathbf{x})|\leq 1$ and, as a consequence (see, e.g., Corollary A.27 in \cite{Fi}),
\begin{equation}\label{eq:partial_bounded_for_1}
\partial \varphi(\mathbb{R}^d)\subseteq \bar{\mathbb{B}}_d;
\end{equation}
here, and throughout this paper,   $\bar{B}$ stands for the closure of a set $B$, $\partial{\varphi}({\bf x})$ for the subdifferential\footnote{Recall that the subdifferential of $\varphi$ at~$\bf x$ is the set of all ${\bf z}\in\mathbb{R}^d$ such that $\varphi({\bf y})-\varphi({\bf x})\geq \langle {\bf z}, {\bf y}-{\bf x} \rangle$ for all $\bf y$.}   of the convex function $\varphi$ at~$\bf x$, and $\partial \varphi(A):=\bigcup_{{\bf x}\in A} \partial \varphi({\bf x})$. Furthermore, Proposition 10  in \cite{McCann} (see also Remark 16) shows that, since $\rm P$ has a density, 
%$\nabla \varphi$ pushes $P$ forward to $\mathrm{U}_d$ and also that
  $\nabla \psi(\nabla \varphi(\mathbf{x}))=\mathbf{x}$ for almost every $\mathbf{x}$ in the support of $\rm P$ and $\nabla \varphi(\nabla \psi(\mathbf{y}))=\mathbf{y}$ for almost every $\mathbf{y}\in \mathbb{B}_d$. In that sense, ${\bf Q}_\pm$ and ${\bf F}_\pm$ are the inverse of each other.
%
%We define
%\begin{equation}\label{defFpm}
%\mathbf{F}_\pm(\mathbf{x})=\nabla \varphi(\mathbf{x}).
%\end{equation}
In this way, we have defined $\mathbf{F}_\pm (\mathbf{x})$ for almost every $\mathbf{x}\in\mathbb{R}^d$ and $\mathbf{Q}_\pm (\mathbf{u})$ for almost every~$\mathbf{u}\in\mathbb{B}_d$;  the definitions coincide with those in \cite{Chernouzhukovetal} or \cite{Hallin}
for $\mathbf{x}$ in the support of $\rm P$.

\subsection{Some regularity results for   Monge-Amp\`ere
equations}\label{Section22}

As announced in the Introduction, our approach to the regularity of the center-outward distribution and quantile functions is based on the classical regularity theory for  Monge-Amp\`ere
equations. We refer to \cite{Fi} for a comprehensive account of this theory, of which we present here a minimal account. 

Given an open set 
${\cal X}\subseteq\mathbb{R}^d$ and a (finite) convex function $\varphi:{\cal X}\to \mathbb{R}$,   denoting by $\ell_d$ he Lebesgue measure on $\mathbb{R}^d$, the Monge-Amp\`ere measure associated with $\varphi$ is  defined by
$$\mu_\varphi(E):=\ell_d\Big(\partial \varphi(E)\Big)$$
for every Borel set $E{\subseteq} \mathbb{R}^d$. It can be checked that $\mu_\varphi$ is indeed a locally finite Borel measure on $\cal X$. The crucial link between center-outward distribution functions and Monge-Amp\`ere measures can be summarized as follows. Assume $\rm P$ is a probability on $\mathcal{X}$ with  Lebesgue density $p$ and let $\varphi$ be a convex function {from $\cal X$ to $\mathbb{R}$}. Then, for every Borel set $A$, 
$${\rm Q}(A):=(\nabla \varphi \sharp {\rm P})(A)={\rm P}\big(\partial{\varphi^*}(A)\big)$$ 
where $\varphi^*$ is the Legendre transform of  $\varphi$. We recall that convexity of $\varphi$ implies that it is differentiable at almost every point in $\mathcal{X}$ (see, e.g., Theorem 25.4 in \cite{Rockafellar}) and, therefore, 
$$(\nabla \varphi \sharp {\rm P})(A)=P(\{\mathbf{x}: \nabla \varphi(\mathbf{x}) \in A\})=P(\{\mathbf{x}: \partial \varphi(\mathbf{x}) \subseteq A\}).$$
 This and the fact that $\mathbf{y}\in \partial \varphi(\mathbf{x})$ if and only if $\mathbf{x}\in \partial \varphi^*(\mathbf{y})$ yield the last equality above. {Hence, if}
%Additionally, if $\nabla \varphi$ pushes  $\rm P$ forward to $\rm Q$ where
 $\rm Q$ has a density $q$,   for every Borel set $A$,
$$\int_{\partial \varphi (A)}q({\bf y})d{\bf y}=\int_A p({\bf x})d{\bf x}$$ 
(see Lemma 4.6 in \cite{Vi}); if, moreover,   ${\rm Q}=\mathrm{U}_d$, %Hence, if %$P$ has a density and
% $\varphi$ is the convex potential (as defined in the Introduction) such\linebreak  that
%  $\nabla \varphi$ 
%pushes $\rm P$ forward to $\mathrm{U}_d$, 
\begin{equation}\label{eq:villani47}
\int_{\partial \varphi(A)} u_d({\bf y})d{\bf y} = \int_{A}p({\bf x})d{\bf x}.% \quad \mbox{for every Borel set } A\subset\R^d.
\end{equation}
Observing  that%, for every Borel set $A\subset\R^d$,
\begin{align*}
\mu_{ \varphi } (A)= \ell_d(\partial \varphi(A))
= \ell_d(\partial \varphi(A)\cap \mathbb{B}_d ),
\end{align*}
where the second equality follows from \eqref{eq:partial_bounded_for_1}, we obtain from \eqref{eq:villani47} that, for $A$ such that~$\ell_d(A)=0$, 
\begin{align*}\label{eq:abs_cont}
\mu_{\varphi} (A) &\leq {a_d}\int_{\partial \varphi(A)}u_d({\bf y})d{\bf y}
= {a_d}\int_{A}p({\bf x})d{\bf x}=0
\end{align*}
with $a_d$ as in \eqref{densityUd}. Thus, the Monge-Amp\`ere measure $\mu_{\varphi}$ is Lebesgue-absolutely continuous.
Since  the density of the absolutely continuous part of the Monge-Amp\`ere measure $\mu_{ \varphi }$ 
is given by~$\big({p({\bf x})}/{u_d(\nabla \varphi( {\bf x}))}\big)$ (see McCann \cite{Mc} or Theorem 4.8 in \cite{Vi}),
we conclude that, for every Borel set $A\subseteq \R^d$, 
\begin{equation}\label{eq:abs_cont_dens}
\mu_{ \varphi } (A)=\int_{A} \frac{p({\bf x})}{u_d(\nabla \varphi( {\bf x}))}d{\bf x}={a_d} \int_{A}p({\bf x})|\nabla \varphi( {\bf x})|^{d-1}d{\bf x}. 
\end{equation}

%On the other hand we can consider 
%\begin{equation}\label{eq:newdef2}
%\tilde{\psi}(z)=\sup_{y\in B_1, \ q\in \partial \psi(y)}\{ \left\langle q,z-y\right\rangle+\psi(y)\}, \ \ z \in \R^d.
%\end{equation}
%It is easy to check that $\tilde{\psi}$ and $\psi$ are a.e. equal in $B_1$. 

Let us focus now on the Monge-Amp\`ere measure $\mu_\psi$ associated (see Section~\ref{Section21}) with  ${\bf Q}_\pm$ and~$\psi$ (both defined over $\mathbb{B}_d$). %In this case we consider $\psi$ as a function defined on $B_1$. 
Since $\nabla \psi$ pushes $\mathrm{U}_d$ forward to $\rm P$, we have that
$\nabla \psi({\bf y})\in {\cal X} $ ${\bf y}$-a.e.~in $\mathbb{B}_d$. By continuity (see Theorem 25.5 in \cite{Rockafellar}), 
%we see that
 $\nabla \psi({\bf y})\in \bar{\mathcal X}$ for every point  ${\bf y}$ 
of differentiability of~$\psi$. Using again Corollary A.27 in \cite{Fi}, we conclude that $\partial \psi({\mathbb{B}}_d)$ is included in the convex hull~$\overline{\text{conv}{({\mathcal X})}}$ of $\mathcal X$.
Hence, if  $\mathcal X$ itself is convex,  we obtain that
\begin{equation}\label{eq:partial_bounded_for_1_phi2}
\partial \psi({\mathbb{B}}_d)
\subseteq
\overline{\mathcal{X}}  . 
\end{equation}
Analogous to \eqref{eq:villani47}, we have that
\begin{equation}\label{eq:villani4.72}
\int_{\partial \psi(B)}p({\bf x})d{\bf x} = \int_{B}u_d({\bf y})d{\bf y} \ \ \text{for every Borel set $B\subseteq\R^d$}. 
\end{equation}
Now, denoting by $r\,\mathbb{B}_d$ the open ball with radius $r$ centered at the origin,  let us assume  that the Borel set $B\subseteq r\,\mathbb{B}_d$, with $0<r<1$, has Lebesgue measure zero. Since~$\bar{B}\subseteq r\,\mathbb{B}_d$ is compact,~$\partial \psi(\bar{B})$ also is compact (see, e.g. Lemma A.22 in \cite{Fi}). Hence, there exists $R>0$ such that 
$$\partial \psi(B)\subseteq \partial \psi(\bar{B}) \subseteq R\,\mathbb{B}_d.$$ 

The following assumption, which requires the density $p$ of $\rm P$ to be bounded  and bounded away from 0 on compact subsets of the support, is absolutely essential (the same assumption is  also made by Figalli  in \cite{Fi2}). \medskip

\noindent {\bf Assumption A. } For every $R>0$,  there exist constants $0<\lambda_R\leq  \Lambda_R$ such that  
%\begin{equation}\label{upperlowwer}
% \lambda_R I[{\bf x}\in {\mathcal X}\cap R\,\mathbb{B}_d]\leq p({\bf x})I[{\bf x}\in {\mathcal X}\cap R\,\mathbb{B}_d]\leq 
% \Lambda_R I[{\bf x}\in{\mathcal X}\cap R\,\mathbb{B}_d].\bigskip
%\end{equation}
\begin{equation}\label{upperlowwer}
 \lambda_R  \leq p({\bf x}) \leq 
 \Lambda_R \quad\text{ for all ${\bf x}\in{\mathcal X}\cap R\,\mathbb{B}_d$}.\bigskip
\end{equation}

Since $\mathcal X$ is convex (hence $\ell_d(\bar{\mathcal X}-{\mathcal X})=0$), Assumption~A entails 
\begin{align*}\label{eq:abs_cont2}
\mu_{ \psi } (B) &\leq \frac{1}{\lambda_R}\int_{\partial \psi(B)}p({\bf x})d{\bf x}
= \frac{1}{\lambda_R}\int_{B}u_d({\bf y})d{\bf y}=0.
\end{align*}
Assuming convexity of $\mathcal X$ and \eqref{upperlowwer}, we conclude  that $\mu_\psi $ is absolutely continuous with respect to $\ell_d$ and, using  Theorem 4.8 in \cite{Vi} again, that,  for every Borel set $B\subset \mathbb{B}_d$, 
\begin{equation}\label{eq:abs_cont_dens2}
\mu_{ \psi } (B)=\int_{B} \frac{u_d({\bf y})}{p(\nabla \psi( {\bf y}))}d{\bf y}=\frac 1{a_d} \int_{B}\frac{1}{p(\nabla \psi({\bf y}))|{\bf y}|^{d-1}}d{\bf y}. 
\end{equation}
We summarize this discussion in the next proposition.
\begin{Proposition}\label{MongeAmpereMeasures}
Let  $\rm P$ be a probability measure with  density $p$ supported on the open set $\mathcal{X}\subseteq \mathbb{R}^d$. Denote by $\psi:\mathbb{B}_d\to \mathbb{R}$ the  convex, lower semicontinuous function satisfying  $\psi({\bf 0})=0$ and $\nabla \psi\sharp \mathrm{U}_d={\mathrm P}$ and let $\varphi:\mathbb{R}^d\to \mathbb{R}$ be defined as in \eqref{defPhi}. Then,
\begin{itemize}
\item[(i)] $\mu_\varphi$ is absolutely continuous with respect to $\ell_d$ and, for every Borel $A\subseteq \mathbb{R}^d$,
$$\mu_{ \varphi } (A)={a_d} \int_{A}p({\bf x})|\nabla \varphi( {\bf x})|^{d-1}d{\bf x};$$
\item[(ii)]  if, moreover, $\mathcal X$ is convex and $p$ satisfies Assumption~A, then $\mu_\psi$ is absolutely continuous with respect to $\ell_d$ and, 
for every Borel set $B\subseteq \mathbb{B}_d$,
$$\mu_{ \psi } (B)=\frac{1}{a_d}\int_{B}\frac{1}{p(\nabla \psi({\bf y}))|{\bf y}|^{d-1}}d{\bf y}.$$
\end{itemize}
\end{Proposition}

Next, let us  show   that,  for well-behaved probability measures $\rm P$ (those with convex support and  density $p$ satisfying Assumption~A),
the center-outward distribution function ${\bf F}_\pm$ cannot map points in the interior of the support of $\rm P$ to extremal points of the unit ball.
\begin{Lemma}\label{lem:boundary}
Let $\rm P$ be a probability measure with  density $p$ supported on the convex open set ${\mathcal X}\subseteq \mathbb{R}^d$ and such that 
Assumption~A holds. Then $(\partial \varphi )(\mathcal{X})\cap \mathbb{S}_{d-1}=\emptyset$, where $\mathbb{S}_{d-1}=\bar{\mathbb{B}}_d\setminus\mathbb{B}_d$.
\end{Lemma}
\begin{proof}
Assume %, on the contrary,
 that there exists $x\in \mathcal{X}$ such that   $|{\bf y}|=1$ for some ${\bf y}\in \partial \varphi ({\bf x})$.  
Without loss of generality, we can assume ${\bf x}={\bf 0}$. Since $\mathcal X$ is open,   there exists $\epsilon>0$ such that~$\epsilon\overline{\mathbb{B}}_{d}\subset \mathcal{X}$. For small~$\theta>0$,  consider the sets
\begin{align*}
	\mathcal{C}_{\epsilon,\vartheta }&:=\Big\{{\bf x}\in \R^d : \Big\vert \frac{{\bf x}}{| {\bf x}|}-{\bf y}\Big\vert\leq \sin{\vartheta} , |{\bf x}|\leq\epsilon\Big\} \\ 
		\mathcal{D}_{\vartheta }&:=\{{\bf b}\in \mathbb{B}_d: \ \langle {\bf y}-{\bf b}, {\bf y}\rangle\leq 2\vartheta|{\bf y}-{\bf b}|\}.
\end{align*}
Now, if ${\bf a}\in \mathcal{C}_{\epsilon,\vartheta }$ and ${\bf b} \in \partial \varphi ({\bf a})$, the monotonicity of $\partial \varphi$
implies that $\langle {\bf y} -{\bf b},{\bf a} \rangle\leq 0$. Hence, 
$$\langle {\bf y}-{\bf b}, {\bf y}\rangle=\langle {\bf y}-{\bf b}, {\bf y}-\frac{{\bf a}}{|{\bf a}|}\rangle+\langle {\bf y}-{\bf b},\frac{{\bf a}}{|{\bf a}|}\rangle
\leq |{\bf y}-{\bf b} |\sin (\vartheta)\leq |{\bf y}-{\bf b} |2\vartheta.$$
This shows that $\partial \varphi(\mathcal{C}_{\epsilon,\vartheta })\subseteq \mathcal{D}_{\vartheta }$.
But the density $p$, inside $\mathcal{C}_{\epsilon,\vartheta }$, is bounded from below  by $\lambda_{\epsilon}$  and the density $u_d$ is bounded from above by $2/a_d$ inside $\mathcal{D}_{\vartheta }$ for $\vartheta\ll 1$: then, in view of the transport equation~\eqref{eq:villani47}, we have
\begin{align*}
	\frac 2{a_d}\ell_d\big(\mathcal{D}_{\vartheta } \big)\geq \int_{\mathcal{D}_{\vartheta }}u_d({\bf b}) d{\bf b}\geq  \int_{\partial \varphi(\mathcal{C}_{\epsilon,\vartheta })}u_d({\bf b}) d{\bf b}= \int_{\mathcal{C}_{\epsilon,\vartheta } }p({\bf x}) d{\bf x}\geq \lambda_{{\epsilon}}\ell_d\big(\mathcal{C}_{\epsilon,\vartheta }\big).
\end{align*}
This, however, cannot hold true since $\ell_d\big(\mathcal{C}_{\epsilon,\vartheta }\big)\approx \epsilon ^d \vartheta^{d-1}$ and $\ell_d\big(\mathcal{D}_{\vartheta }\big)\approx \vartheta^{d+1}$ as $\theta\to 0$. The claim follows. 
\end{proof}

{
We now proceed   to provide sufficient conditions under which the center-outward quantile function~${\bf Q}_\pm$ is continuous at every point in the
open unit ball (except, possibly, at the origin).
It is well known that differentiability of a lower semicontinuous convex function $\psi$ (which entails continuity of its gradient) is equivalent to strict convexity of its convex conjugate (see Theorem 26.3 in \cite{Rockafellar}). As announced, the techniques we are using here are in the spirit of those developed by Caffarelli in \cite{Ca1}, \cite{Ca2} or Figalli in \cite{Fi}, \cite{Fi2}, which in turn  largely rely on the fact that, under some control for the Monge-Amp\`ere measure, the intersection between the graph and supporting hyperplanes of $\psi$ either consists of a single point or has an extreme point (see Theorem 4.10 in \cite{Fi}). A central result in Caffarelli's regularity theory (see Corollary 4.21 in \cite{Fi}) is that a strictly
convex function $\psi$ on an open set $\Omega$ for which there exist constants $0<\lambda<\Lambda$ such that 
\begin{equation}\label{MGbounds}
\lambda\ell_d(A)\leq\mu_\psi(A)\leq \Lambda\ell_d(A)
\end{equation}
for every Borel set $A\subseteq \Omega$ is automatically of class $\mathcal{C}_{\mbox{\scriptsize loc}}^{1,\alpha}$ for some $\alpha>0$ that depends only on $\lambda, \Lambda$, and $d$ (condition (\ref{MGbounds}) in the sequel  will be summarized, with a slight abuse of notation, as $\lambda d\mathbf{x}\leq \mu_{\psi}\leq \Lambda d\mathbf{x}$).
The fact that 
$\mathrm{U}_d$ for~$d\geq 2$ has an unbounded density  adds some complication to the particular problem here, though. On the other hand, the density 
$u_d$ is bounded away from 0, which allows to control the growth of the Monge-Amp\`ere measure, as we show next.
%\color{black}
\begin{Lemma}\label{lem:bounded}
If $\rm P$ satisfies the assumptions in Proposition \ref{MongeAmpereMeasures}{\it (ii)}, denoting by  $M$ a compact subset of~$\mathbb{B}_d$,  there exist constants $\alpha_M$ and $A_M$ such that, 
for every Borel set $A\subseteq M$,
\begin{align}
	\alpha_M\ell_d(A)\leq \mu_{\psi}(A)\leq A_M(\ell_d(A))^{1/d}.
\label{lem23}\end{align}
\end{Lemma}
\begin{proof}
The compactness of $M$ entails that of $\partial \psi(M)$;  in particular,  $\partial \psi(M)\subseteq R\,\mathbb{B}_d$ for some $R>0$.
Hence, using Proposition \ref{MongeAmpereMeasures}{\it (ii)} and taking $\lambda_R, \Lambda_R\in \R$ as in Assumption~A, we obtain 
\begin{align*}
	\mu_{ \psi } (A)=a_d \int_{A}\frac{1}{p(\nabla \psi({\bf y}))|{\bf y}|^{d-1}}d{\bf y}\geq \frac{a_d}{\Lambda_R}\ell_d(A) . 
\end{align*}
For the upper bound in \eqref{lem23}, note that the ball ${(\ell_d(A)/c_d)^{1/d}}B_d$ (where {$c_d=\pi^{d/2}/\Gamma(1+d/2)$} denotes the volume of the $d$-dimensional
unit ball) maximizes $\int_B |{\bf y}|^{1-d} d{\bf y}$ among all subsets of $\mathbb{B}_d$ with  Lebesgue measure   $\ell_d(A)$.
On the other hand, by the co-area formula (see, e.g., Proposition 1, p. 118 in \cite{EvansGariepy}), 
\begin{equation}
\int_{r\mathbb{B}_d} |{\bf y}|^{1-d} d{\bf y}=\int_0^r \left[ \int_{\partial s\mathbb{B}_d} |{\bf y}|^{1-d} d
\mathcal{H}^{d-1}({\bf y})\right] ds=\int_0^r a_d\, ds=a_d\, r\label{coarea}
\end{equation}
where $\mathcal{H}^{d-1}$ denotes the $(d-1)$-dimensional Hausdorff measure. 
Combining \eqref{coarea}  with Proposition~\ref{MongeAmpereMeasures}{\it (ii)}, we conclude that
$$\mu_{\psi}(A)\leq \frac{1}{\lambda_R\, a_d} \int_A |{\bf y}|^{1-d}d{\bf y}\leq \frac{1}{\lambda_R\, a_d^{1/d}} (\ell_d(A))^{1/d}.\vspace{-11mm}$$
\end{proof}
\medskip

%When a convex function $f$ is not strictly convex in a point $x\in\text{dom}(f)$, then the convex set $\Sigma$ defined in Lemma \ref{lem:singleton} is not a singleton. In order to continue the steps of the Caffarelli Figalli theory, we will start discarding the places where the set $\Sigma$ has exposed point. Then it is the moment to generalize the well known result \cite[Theorem 4.10]{Fi}.

%After this result we are able to prove the main theorem of the section. 
Note that 
the lower bound in Lemma~\ref{lem:bounded} remains valid for a compact subset $M$ of $\bar{\mathbb{B}}_d$ provided that~$\partial{\psi}(M)$ is bounded: indeed, 
 that lower bound only requires the upper bound from Assumption~A. A similar conclusion holds for the upper bound. Additionally, if the density $p$ is uniformly bounded, the lower bound holds for any subset of $\bar{\mathbb{B}}_d$.

\subsection{Main result}\label{Section23}

We are ready now for the main result of this note. 
Our proof follows the lines of \cite{Fi}, \cite{Fi2}, and \cite{Fi3}, but we cover cases in which the range of the extension to the whole space of the  center-outward quantile function is not necessarily equal to $\R^d$. As in the last reference, we have to handle carefully the fact that $\mathcal X$ is not {necessarily} bounded and use a ``minimal'' extension of the quantile function potential, namely,
\begin{equation}\label{eq:newdef2}
\tilde{\psi}({\bf z}):=\sup_{{\bf b}\in \mathbb{B}_d, \ {\bf y}\in \partial \psi({\bf b})}\{ \left\langle {\bf y},{\bf z}-{\bf b}\right\rangle+\psi({\bf b})\}, \ \ {\bf z} \in \R^d.
\end{equation}
Obviously, $\tilde{\psi}$ is still a lower semicontinuous convex function and $\tilde{\psi}({\bf z})$ coincides with~$\psi({\bf z})$ for~${\bf z}\in\mathbb{B}_d$. Since $\mathbf{Q}_\pm({\bf z}):=\nabla \psi({\bf z})\in \mathcal{X}$ for every 
differentiability point ${\bf z}$ of $\psi$ in $\mathbb{B}_d$, we see  (using, once more, Corollary A.27 in \cite{Fi}) that, provided that $\mathcal X$ is convex, $\partial \psi(\mathbb{B}_d)\subseteq \bar{\mathcal X}$.
The ``minimality'' of the extension \eqref{eq:newdef2} refers to the fact that 
%\begin{equation}\label{eq:contained}
$\partial \tilde{\psi}(\R^d)\subset \bar{\mathcal X}$, 
%\end{equation}
as can be checked from a simple application of the Hahn-Banach separation theorem. Of course, the values of $\tilde{\psi}$ outside $\mathbb{B}_d$ are not relevant for the study of its differentiability inside $\mathbb{B}_d$, but the use of $\tilde{\psi}$ will be useful in the next proof. We note also that the discussion leading to Proposition \ref{MongeAmpereMeasures} can be reproduced with $\tilde{\psi}$ substituted for $\psi$ to conclude that $\mu_{\tilde{\psi}} $ is absolutely continuous  with respect to the Lebesgue measure and that, for every Borel set $B\subseteq \R^d$, 
\begin{equation}\label{eq:abs_cont_dens3}
\mu_{\tilde{\psi}} (B)=\int_{B\cap {\mathcal X}} \frac{u_d({\bf y})}{p(\nabla \tilde{\psi}( {\bf y}))}d{\bf y} .
\end{equation}
Finally,  observe that $\mu_{\tilde{\psi}}$ in concentrated on $\mathbb{B}_d$, that is, if $B\subseteq \R^d\setminus \mathbb{B}_d$, then $\mu_{ \tilde{\psi} } (B)=0$, 
see Theorem~4.8 in \cite{Vi} or \cite{Fi3} for further details.\medskip

The main result of this note follows from the following crucial lemma.

\begin{Lemma}\label{lem:convex} Under the assumptions of Theorem~\ref{th:homeomorphism}, 
$\tilde{\psi}$ is  strictly convex   on $\mathbb{B}_d$. 
\end{Lemma}

\begin{proof}
To prove this, assume   that the contrary holds true. Then, there exists ${\bf y}\in \mathbb{B}_d$ and 
${\bf t}\in\partial \tilde\psi({\bf y})$ such that, putting $l({\bf z}):=\tilde{\psi}({\bf y})+\left\langle {{\bf t}}, {\bf z}-{\bf y}\right\rangle$, the convex   set~$\Sigma:=\{{\bf z}:\,  \tilde{\psi}({\bf z})=l({\bf z})\}$ is not a singleton. By subtracting an affine function, we can assume~$\tilde{\psi}({\bf y})~\!=~\!0$ and~$\tilde{\psi}({\bf z})\geq 0$ for all ${\bf z}$; then, {
$\Sigma=\{{\bf z}:\, \tilde{\psi}({\bf z})=0\} = \{{\bf z}:\, \tilde{\psi}({\bf z})\leq 0\}$}, which  is closed since $\tilde{\psi}$ is lower semicontinuous. Also, by adding the convex function $w({\bf z}):=\frac 1 2 (|{\bf z}|-1)_+^2$ (note that~$\tilde{\psi}=\tilde{\psi}+w$ on $\bar{\mathbb{B}}_d$), 
we can assume that $\Sigma\subset \bar{\mathbb{B}}_d$. Being compact and convex,  $\Sigma$ equals the closed convex hull of its extreme points; as a consequence,  it 
 must have at least two exposed points (otherwise it would be empty or a singleton). {Let $\bar{{\bf y}}\in\bar{\mathbb{B}}_d\setminus\{\mathbf{0}\}$ be one of them.}
If~$\bar{{\bf y}}\in \mathbb{B}_d\setminus\{\mathbf{0}\}$, we consider a small ball $C_{\bar{{\bf y}}}$, say, around $\bar{{\bf y}}$,   such that $\bar{C}_{\bar{\bf y}}\subset {\mathbb{B}}_d\setminus\{\mathbf{0}\}$. Then $\partial \tilde\psi( \bar{C}_{\bar{\bf y}})$ is a compact set, and hence~$\partial \tilde\psi( \bar{C}_{\bar{{\bf y}}})\subset R\,\mathbb{B}_d$ for some $R>0$. By Proposition \ref{MongeAmpereMeasures}{\it (ii)}, we have  constants~$0<\lambda_{C_{\bar{{\bf y}}}}\leq \Lambda_{C_{\bar{{\bf y}}}}$ such that the Monge-Amp\`ere measure $\mu_{\tilde{\psi}}$ satisfies 
$\lambda_{C_{\bar{{\bf y}}}}d{\bf x}\leq \mu_{\tilde{\psi}} \leq \Lambda_{C_{\bar{{\bf y}}}}d{\bf x}$ in~$C_{\bar{{\bf y}}}$. But the set $\Sigma$ has an exposed point in $C_{\bar{{\bf y}}}$ and this contradicts Theorem~4.10 in \cite{Fi}. Consequently, we must assume that~$\bar{{\bf y}}\in \partial \mathbb{B}_d$. Observe that~$\tilde{\psi}(\bar{{\bf y}})~\!=~\!0$, hence~$\bar{{\bf y}}\in \mbox{dom}(\tilde{\psi})$. First consider  the case where~$\bar{{\bf y}}\notin \partial\big(\mbox{dom}(\tilde{\psi})\big)$. Let~$\mathbb{B}_r({\bf x}):={\bf x}+ r\mathbb{B}_d$ and~$\bar{\mathbb{B}}_r({\bf x})$ denote, respectively  the open and the closed ball of radius $r$ centered at $\bf x$. Then, for  $\eta>0$ small enough,~$\bar{\mathbb{B}}_\eta({\bf y}) \subset \mbox{dom}(\tilde{\psi})$; consequently,
there exists some $R_0$ such that $\partial \tilde{\psi}\big(\bar{\mathbb{B}}_\eta(\bar{\bf y})\big)\subset R_0\, \mathbb{B}_{d}$. For $\eta$ small enough, we further can ensure that~$\bar{\mathbb{B}}_\eta(\bar{\bf y})\subset 2\mathbb{B}_{d}$.

Without any loss of generality, let us assume  that $\bar{\bf y}=\textbf{e}_1$ where $\textbf{e}_1$ stands for the first vector in the canonical basis of $\mathbb{R}^d$  (we can use a rotation otherwise):
$$\Sigma\subset \big\lbrace {\bf z}=(z_1,\ldots,z_d)^\prime\in\R^{d}:\ z_1\leq 1\big\rbrace , \quad\text{and}\quad 
\Sigma\cap \big\lbrace {\bf z}=(z_1,\ldots,z_d)^\prime \in \R^{d}:\ z_1=1 \big\rbrace=\{ \textbf{e}_1 \}.$$  
For $\sigma\in(0,1)$ small enough, we have 
$$\Sigma \cap \big\lbrace {\bf z}\in \R^d:\ z_1\geq 1-\sigma \big\rbrace\subset \mathbb{B}_d\cap \big\lbrace {\bf z}\in \R^d:\ z_1\geq 1-\sigma \big\rbrace
\subset  {\mathbb{B}}_{\eta}(\textbf{e}_1).$$
For such $\sigma$,  defining
\begin{align}
\psi_{\epsilon}({\bf z}):=\tilde{\psi}({\bf z})-\epsilon(z_1-1+\sigma)\quad\text{and}\quad S_{\epsilon}:=\{{\bf z}:  \psi_{\epsilon}({\bf z}) < 0\},
\end{align}
observe that
\begin{align}\label{Hconv}
S_{\epsilon}\longrightarrow \Sigma \cap \big\lbrace {\bf z}\in \R^d:\ z_1\geq 1-\sigma \big\rbrace
\end{align}
in the Hausdorff distance\footnote{Recall that, for $A,B\,\subseteq \mathbb{R}^d$,  $\mathrm{d}_H(A,B):= \max\big\{\sup_{{\bf a}\in A}\inf_{_{\scriptstyle{\bf b}\in B}}\vert {\bf a}-{\bf b}\vert ,\,\sup_{{\bf b}\in B} \inf_{_{\scriptstyle{\bf a}\in A}}\vert {\bf a}-{\bf b}\vert \big\}.$} $\mathrm{d}$ as $\epsilon\to 0$. Hence, for $\epsilon >0$ small enough, the sets $S_{\epsilon}$ are bounded open convex subsets of the ball $\mathbb{B}_{\eta}(\textbf{e}_1)$. By Lemma~\ref{lem:bounded},  there exists some $M>0$ such that $$\mu_{\psi_\epsilon}(A)=\mu_{\tilde{\psi}}(A)\leq M
(\ell_d(A))^{1/d}$$ for every $A\subset S_\epsilon$ and  $\epsilon$ small enough.

Next,  fix ${\bf z}_0\in \mathbb{B}_d\cap \Sigma$ and $\delta>0$ such that $\bar{\mathbb{B}}_{\delta}({\bf z}_0)\subset \mathbb{B}_d\cap \mathbb{B}_\eta(\mathbf{e}_1)$
and
consider the normalizing map~$L_{\epsilon}$---namely, the affine transformation $L_{\epsilon}$ that normalizes\footnote{{A convex set $\Omega\subset \R^d$ is said to be normalized if $\mathbb{B}_d\subseteq \Omega \subseteq d\,\mathbb{B}_d$. For each open bounded convex set $\Omega$ there exists a unique invertible affine transformation $L$   normalizing $\Omega$ (this is  John's  celebrated  Lemma of convex analysis, see Lemma A.13 in \cite{Fi}). We refer to $L$ as the {\it normalizing map} and to   $L(\Omega)$ as the {\it normalized version} of $\Omega$.}} $S_{\epsilon}$; denote by $v_{\epsilon}$ the {normalized solution in $S_{\epsilon}^L:=L_{\epsilon}(S_{\epsilon})$ of $\mu_{v_\epsilon}=f\circ L_{\epsilon}^{-1}$ with the boundary condition $v_\epsilon=0$ on $\partial S_{\epsilon}^L$} ($v_\epsilon$ is the convex map that has {Monge-Amp\`ere measure $d\mu_{v_\epsilon}(\mathbf{x})=f\circ L_{\epsilon}^{-1}(\mathbf{x})d\mathbf{x}$} in $S_{\epsilon}^L$ and vanishes at the boundary of $S_{\epsilon}^L$; its existence and uniqueness is guaranteed, for instance, by Proposition 4.2 in~\cite{Fi}).
Since  $\mathbb{B}_d \subset S_{\epsilon}^L $, we have that $ L_{\epsilon}^{-1}(\mathbb{B}_d)\subset 2\mathbb{B}_d$ and,  therefore,  the map $L^{-1}$ satisfies
$$ \left|L_{\epsilon}({\bf x})-L_{\epsilon}({\bf z})\right| \geq \frac{1}{2}|{\bf x}-{\bf z}| \quad \text{for all}\quad {\bf x}, {\bf z} \in \mathbb{R}^{d}.$$
This implies that
\begin{equation}\label{chain}
{L_{\epsilon}(\mathbb{B}_d) { \supset} L_{\epsilon}\left(B_{\delta}\left({\bf z}_{0}\right)\right) \supset B_{\delta / 2}\left(L_{\epsilon} ({\bf z}_{0})\right)}.
\end{equation}
We consider the sets $ S_{\epsilon, \delta}^{L}:=\left\{{\bf z} \in S_{\epsilon}^{L}: \mathrm{d}_H\left({\bf z}, \partial S_{\epsilon}^L\right) \geq \delta /4\right\} $.
Now $L_{\epsilon}({\bf z}_{0})\in S_{\epsilon}^{L}$, a  normalized set (it contains the unit ball and is contained in the ball {of radius $d$, the dimension of the Euclidean space}). This implies that there exists a constant  $k_d>0$, depending only
on $d$ such that  (see Theorem 4.23 in \cite{Fi} or Lemma 3 in \cite{Ca3})
$$\ell_d\big(S_{\epsilon, \delta}^{L} \cap \mathbb{B}_{\delta /2}(L_{\epsilon} ({\bf z}_{0}))\big)\geq k_d  \delta^d. $$  In view of Lemma \ref{lem:bounded}, the subsequent remark, and the fact that $\bar{\mathbb{B}}_{\delta}({\bf z}_0)\subset \mathbb{B}_\eta(\mathbf{e}_1)$, we have that~$\mu_{\psi_\epsilon}$
is lower bounded over $\bar{\mathbb{B}}_{\delta}({\bf z}_0)$, that is, there exists $\lambda>0$ such that $\mu_{\psi_\epsilon}(A)\geq \lambda \ell_d(A)$
for every~$A~\!\subseteq~\!\bar{\mathbb{B}}_{\delta}({\bf z}_0)$. This and \eqref{chain} thus imply that $\mu_{v_{\epsilon}}$ is bounded from below  on $\mathbb{B}_{\delta /2}(L_{\epsilon} ({\bf z}_{0}))$. 
%Hence, since the Monge-Amp\`ere measure $\mu_{v_{\epsilon}}$ is bounded from below  on $\mathbb{B}_{\delta /2}(L_{\epsilon} ({\bf z}_{0}))$ (from 
%Lemma \ref{lem:bounded}, the subsequent remark, and the fact that $\bar{B}_{\delta}({\bf z}_0)\subset \mathbb{B}_\eta(\mathbf{e}_1)$), we have that $\mu_{\psi_\epsilon}$
%is lower bounded in $\bar{B}_{\delta}({\bf z}_0)$; \eqref{chain} completes the proof of this claim), 
It follows that, {for some $\lambda >0$,}
$$\mu_{{v}_{\epsilon}}\left(S_{\epsilon, \delta}^{L}\right) \geq \lambda \ell_d\Big(S_{\epsilon, \delta}^{L} \cap \mathbb{B}_{\delta /2}\left(L_{\epsilon} ({\bf z}_{0})\right)\Big) \geq C\delta^{d}.$$
This implies that, for  $c'$ small enough, no ball of radius $ c'  {\delta}/{2}$ can contain $\partial v_{\epsilon}\big(S_{\epsilon, \delta}^{L}\big)$. 
% is not contained in any ball of radius $ c' \frac{\delta}{2}$,  for $c'$ small enough and, 
 As a consequence,  there exists $c>0$ such that $ \sup _{{\bf p} \in \partial v_{\epsilon}\left(S_{\epsilon, \delta}^{L}\right)}|{\bf p}| \geq c {\delta}$. 
Using Corollary A.23 in \cite{Fi}, we conclude that
$$ {\big|\min _{S_{\epsilon}^L} v_{\epsilon}\big| \geq c^{\prime \prime}\left({\textstyle  {\delta}/{2}}\right)^{2}}$$
for some $c''>0$. On the other hand, {using Lemma~\ref{lem23} again to upper bound} $\mu_{\psi_\epsilon}$, we obtain 
$$\mu_{v_{\epsilon}}\left(S_{\epsilon}^L\right)=\mu_{\psi_\epsilon} \left(S_{\epsilon}\right)\leq M (\ell_d(2\mathbb{B}_d))^{1/d}$$
and, by  the Alexandrov maximum principle (e.g. Theorem 2.8. in \cite{Fi}), this implies that 
$$\left|v_{\epsilon}\left(L_{\epsilon} \textbf{e}_1\right)\right| \leq C \left(\mathrm{d}_H (L_{\epsilon} \mathbf{e}_1, \partial  S_{\epsilon}^L)\right)^{1 / d}.$$
This means that   the same arguments as in the proof of {Theorem 4.10 in \cite{Fi}} yield 
$$ \lim_{\epsilon\rightarrow 0}\mathrm{d}_H\left(L_{\epsilon} \mathbf{e}_1, \partial( S_{\epsilon}^L) \right)= 0\quad\text{and}\quad  \lim_{\epsilon\rightarrow 0}\frac{v_{\epsilon}\left(L_{\epsilon} \mathbf{e}_1\right)}{\min _{ S_{\epsilon}^L} v_{\epsilon}} = 1,$$
which is a contradiction. 

Finally,  consider the case where the exposed point of $\Sigma$ belongs
to $\partial \big( \mbox{dom}(\tilde{\psi})\big)$; here again, it can be assumed  that $\Sigma=\{{\bf z}: \tilde{\psi}({\bf z})=0\}$ and that the exposed point of $\Sigma $ is  $\textbf{e}_1$.
We also can   assume, without loss of generality, that $\mbox{dom}(\tilde{\psi})\subset \big\lbrace {\bf z}\in\R^{d}:\ z_1\leq 1\big\rbrace$. Hence, $\{c\,\textbf{e}_1:\, c\geq 0\}\subset \partial \tilde{\psi}(\textbf{e}_1)$. For small $\theta>0$ we consider the sets 
$$A_\theta:=\mathbb{B}_d\cap\{{\bf z}=(z_1,{\bf z}')\in\R\times\R^{d-1}:\, -\theta|{\bf z}'|\leq z_1-1\leq 0 \}$$
and
$$C_\theta:={\mathcal X}\cap\{{\bf x}=(x_1,{\bf x}')\in\R\times\R^{d-1}:\, x_1> 0,\, |{\bf x}'|\leq \theta x_1 \}.$$
Let ${\bf x}\in C_\theta$ and ${\bf z}\in\partial\tilde{\psi}^*({\bf x})$. Then ${\bf x}\in\partial\tilde{\psi}({\bf z})$ and, thanks to the  monotonicity of $\partial \tilde{\psi}$, we have 
that~$\langle {\bf x}-t\,\textbf{e}_1,{\bf z}-\textbf{e}_1 \rangle\geq 0$ for every $t\geq 0$, which entails $\langle {\bf x},{\bf z}-\textbf{e}_1 \rangle\geq 0$
(take $t=0$) and~$\langle \textbf{e}_1,{\bf z}-\textbf{e}_1 \rangle\geq 0$ (take $t\to \infty$). This means that $z_1\leq 1$ and $x_1(z_1-1)+\langle {\bf x}',{\bf z}'\rangle\geq 0$,
from which we deduce that $z_1-1\geq \theta|{\bf z}'|$. Since, by Lemma \ref{lem:boundary}, we have ${\bf z}\in \mathbb{B}_d$, it follows that $\partial\tilde\psi (A_\theta)\supset C_\theta$.
Also, since both~$\bf 0$ and $\textbf{e}_1$ belong to $\partial \tilde{\psi}(\R)^d\subset \bar{\mathcal X}$, which is a convex set with nonempty interior, we can argue as in pp. 8-9 of \cite{Fi3},  to conclude that $\ell_d({C_\theta}\cap 2\mathbb{B}_d)\gtrsim \theta^{d-1}$ for  $\theta>0$ small enough. From the transport equation,  
we have that 
\begin{equation}\label{lastcontradiction}
\mathrm{U}_d(A_\theta)= \int_{\partial \tilde{\psi}(A_\theta)} p({\bf x}) d{\bf x}\geq \int_{C_\theta} p({\bf x}) d{\bf x}\gtrsim\ell_d({C_\theta}\cap 2\mathbb{B}_d)\gtrsim \theta^{d-1},
\end{equation}
where we have used that $p$ is lower bounded on bounded subsets of $\mathcal X$. However, for small $\theta$, $A_\theta$ is well separated from $\bf 0$ and, consequently, $u_d$ is upper bounded {on $A_\theta$}. This means that
$$\mathrm{U}_d(A_\theta)\lesssim \ell_d(A_\theta)\lesssim \theta^{d+1},$$
which contradicts \eqref{lastcontradiction}. This completes the proof of the claim that $\tilde{\psi}$ (equivalently, $\psi$) is strictly convex in $\mathbb{B}_d$.
\end{proof}

We now can state and, based on Lemma \ref{lem:convex}, prove our main result, which extends Theorem~1.1 in Figalli  \cite{Fi2}  to the case of (bounded or unbounded) convexely supported distributions.

\begin{Theorem}\label{th:homeomorphism} Let  $\rm P$ be a probability measure with  density $p$ supported on the open convex set~$\mathcal{X}~\!\subseteq~\!\mathbb{R}^d$. 

(i) If $p$ satisfies \eqref{upperlowwer},  there exists a compact convex set  $K$  with Lebesgue measure $0$ such that  the center-outward quantile function $\mathbf{Q}_\pm:=\nabla \psi$ and the center-outward distribution function $\mathbf{F}_\pm:=\nabla \psi^*$ are homeomorphisms between  $\mathbb{B}_d\setminus \{ {\bf 0}\}$ and ${\mathcal X}\setminus K$, inverses of each other.%, ${\mathbf F}_\pm$. 
%If it is assumed that $X$ is convex, then $P\in \mathcal{F}$. 

(ii) If, moreover,  $p\in\mathcal{C}_{\text{{\rm loc}}}^{k, \alpha}({\mathcal X}) $ for some  $k\in \N$ and $\alpha\in (0,1)$,  then $\mathbf{Q}_\pm$ and $\mathbf{F}_pm$ are diffeomorphisms of class $\mathcal{C}_{\text{{\rm loc}}}^{k+1, \alpha}$ between  $\mathbb{B}_d\setminus \{ {\bf 0}\}$ and ${\mathcal X}\setminus K$.
% Furthermore it satisfies
%\begin{align}\label{eq:relationtheo}
%\det D^2 \psi (b) = \frac{c_d}{p(\nabla \psi(b))|b|^{d-1}}, \ \ b\in B_1\setminus\{0\}
%\end{align}
\end{Theorem}

%\textcolor{red}{HERE START WITH THE PROOF OF THEOREM~\ref{th:homeomorphism}}
\begin{proof} 
%\begin{proof}
Assumption~A implies, for any closed ball   $B\subseteq \mathbb{B}_d\setminus\{\mathbf{0}\}$, the existence of  constants $0~\!<~\!\lambda_B~\!\leq~\!\Lambda_B$ such that 
{$\lambda_B d{\bf x}\leq \mu_{\tilde{\psi}} \leq \Lambda_B d{\bf x}$}. It follows from Caffarelli's regularity theory (see Corollary 4.21 in~\cite{Fi})  that $\psi$ is locally of class $C^{1,\alpha}$.
The constant $\alpha>0$ depends on $\lambda$ and $\Lambda$ and, consequently, we cannot conclude that,  for some $\alpha>0$,  $\psi\in C^{1,\alpha}_{\text{loc}}(\mathbb{B}_d\setminus\{\mathbf{0}\})$. However, $\psi$ is continuously differentiable on~$\mathbb{B}_d\setminus\{\mathbf{0}\}$ and, therefore, $\mathbf{Q}_\pm=\nabla \psi$ is a (single-valued) continuous function on $\mathbb{B}_d\setminus\{\mathbf{0}\}$.
Furthermore,   the  strict convexity of $\psi$ (Lemma~\ref{lem:convex}) implies that $\mathbf{Q}_\pm$ is injective. By Brouwer's theorem on invariance of domain {(see, e.g., Theorem 2B.3, p. 172 in \cite{Hatcher})},  $\mathbf{Q}_\pm(\mathbb{B}_d\setminus\{\mathbf{0}\})$ is open and 
$\mathbf{Q}_\pm$ is a homeomorphism between~$\mathbb{B}_d\setminus\{\mathbf{0}\}$ and $\mathbf{Q}_\pm(\mathbb{B}_d\setminus\{\mathbf{0}\})$. But then, necessarily, $\mathbf{Q}_\pm(\mathbb{B}_d\setminus\{\mathbf{0}\})={\mathcal X}\setminus K$, which  completes the proof of part {\it (i)} of the theorem.  

Turning to part {\it (ii)},  assume that $p\in\mathcal{C}_{\text{{\rm loc}}}^{k, \alpha}(X) $, fix $\mathbf{x}\in \mathbb{B}_d$, and  consider  a neighbourhood $V$ of~$\mathbf{x}$ such that its closure $\bar V$ is contained in  $\mathbb{B}_d\setminus\{\mathbf{0}\}$. Now, $\psi$ is strictly convex over~$V$  
and there exist constants~$0<\lambda_V<\Lambda_V$ such that $\lambda_V d{\bf x}\leq \mu_{\tilde{\psi}} \leq \Lambda_V d{\bf x}$ on $V$.
Note that $u_d\in\mathcal{C}_{\text{{\rm loc}}}^{k, \alpha}(V) $ for every $k$ and $\alpha$. Hence, we can apply Remark 4.44 in \cite{Fi} to conclude that $\nabla \psi \in \mathcal{C}_{\text{{\rm loc}}}^{k+1, \alpha}(V)$. This completes the proof.
\end{proof}%\qed

To conclude this section,  observe that the center-outward quantile function of a probability measure~$\rm P$ satisfying the assumptions of Theorem \ref{th:homeomorphism} may fail to be continuous (singled-valued) at the origin. However, the center-outward distribution function is single-valued (and consequently continuous) at every point of the support of $\rm P$), since the points in the set $K$ which had to be removed to guarantee that $\mathbf{Q}_\pm$ is a homeomorphism are all mapped by $\mathbf{F}_\pm$ to the origin.

\section{Some further properties of center-outward distribution and quantile functions.}\label{Section3}

To conclude this note, we present three results that more or less directly follow as consequences of Theorem~\ref{th:homeomorphism}. %related to center-outward distribution and quantile functions. 
The first one is about the 
asymptotic invariance of  center-outward distribution functions; the second one deals with  the ability of center-outward quantile functions to capture the shape of a convex supporting set; the third one is a result on the shape of quantile contours, which turn out to satisfy a kind of relaxed version of convexity, connected to the so-called ``lighthouse convexity'' property (see, e.g., pp. 263-264 in \cite{Cholaquidadisetal}).

% that the outer
%quantile contours exhibit under general assumptions. With respect to the invariance, we note that

A classical univariate distribution function $F$
trivially satisfies 
$$\lim_{x\to-\infty} F(x)=0\quad\text{and}\quad \lim_{x\to\infty} F(x)=1,$$
hence, in terms of  the univariate center-outward distribution function $F_\pm:=2F-1$, % the above limits can be rewritten as
$$\lim_{t\to\infty} F_{\pm}(tu)=u\qquad\text{for all $u$ such that $|u|=1$.}$$
%for every $u$ with $|u|=1$.
 Let us show that this carries over to ${\bf F}_\pm$ in general dimension. Keeping the notation from the previous sections, we establish the following result. % in the Introduction for $\varphi$. Our next result gives a slightly more general version of the announced result.

\begin{Proposition}\label{lem:infinito}
Let the probability measure $\rm P$ have a     density on $\mathbb{R}^d$. For any $\bf u$ on the unit sphere~$\mathbb{S}_{d-1}$, any sequence $(t_n)_{n\in \N}$ of real numbers such that $t_n\to\infty$, and any ${\bf y}_n\in \partial \varphi(t_n {\bf u})$, 
 $$
\lim_{n\rightarrow \infty }{\bf y}_n={\bf u}.$$
\end{Proposition}

\begin{proof}
It follows from \eqref{eq:partial_bounded_for_1}   that ${\bf y}_n\in \bar{\mathbb{B}}_d$. Hence, by compactness, there exists  a subsequence along which 
${\bf y}_n \rightarrow {\bf y}_\infty\in\bar{\mathbb{B}}_d$. 
On the other hand, monotonicity of the subdifferential implies that, for all~${\bf x}\in \R^d$ and ${\bf y}\in \partial \varphi ({\bf x})$,
\begin{align*}
\langle {\bf y} -{\bf y}_n, {\bf x}- t_n {\bf u}\rangle\geq 0
\end{align*}
or, equivalently, for all ${\bf x}\in \R^d$ and ${\bf y}\in \partial \psi ({\bf x})$, 
\begin{align*}
\langle {\bf y}-{\bf y}_\infty, {\bf x}\rangle+ \langle  {\bf y}_\infty -{\bf y}_n, {\bf x}\rangle\geq t_n \left( \langle {\bf y}-{\bf y}_\infty,{\bf u}\rangle+\langle {\bf y}_\infty-{\bf y}_n, {\bf u}\rangle\right).
\end{align*}
Fixing $\epsilon >0$ and $N=N(\epsilon)$ such that $|{\bf y}_n-{\bf y}_\infty |<\epsilon$ for all $n\geq N$, we obtain
\begin{align*}
\langle {\bf y}-{\bf y}_\infty, {\bf x}\rangle+ \epsilon |{\bf x}|\geq t_n \left( \langle {\bf y}-{\bf y}_\infty,{\bf u}\rangle-\epsilon \right).
\end{align*}
Hence, for $n$ large enough,  $\langle {\bf y}-{\bf y}_\infty,{\bf u}\rangle-\epsilon<0$ for all ${\bf x}\in \R^d$ and $ {\bf y}\in \partial \varphi ({\bf x})$. Since $\epsilon>0$ is arbitrary, we conclude that
$$\partial \varphi (\R^d)\subset S:=\{{\bf y}: \ \langle {\bf y}-{\bf y}_\infty,{\bf u}\rangle\leq 0\},$$
which is a hyperplane. Now, the fact that $\nabla \varphi$ pushes $\rm P$ forward to $\mathrm{U}_d$ implies that 
$\partial \varphi (\R^d)$ contains almost every ${\bf x}\in \mathbb{B}_d$.
Hence, $\mathbb{B}_d\subset S$, which only can happen if ${\bf y}_\infty={\bf u}$. 
\end{proof}

Under additional smoothness assumptions on $\rm P$,  the announced result for ${\bf F}_\pm$ follows as a corollary.
\begin{Corollary}
Let $\rm P$ satisfy the assumptions in Proposition \ref{MongeAmpereMeasures}{\it (ii)}. Then, for any $\bf u$ on the unit sphere~$\mathbb{S}_{d-1}$ and any sequence $(t_n)_{n\in \N}$ of real numbers such that $t_n\to\infty$, we have
$$\lim_{n\to\infty} {\bf F}_{\pm}(t_n{\bf u})={\bf u}.$$
\end{Corollary}

Next, we include the announced simple result showing that the outer quantile contours of a convexely supported $\rm P$ approach (in Hausdorff distance) the boundary of its  support.

\begin{Lemma}\label{lem:CONVERGENCE}
%Let $U_d$ be the uniform measure on $\mathbb{B}_d$, and let $P=p(x)dx$ be a probability measure on $\R^d$
Let $\rm P$ be a probability measure   on $\mathbb{R}^d$ with compact convex support  ${\mathcal X}$ and a density~$p$ such that $\lambda\leq p \leq \Lambda$ for some $0<\lambda\leq \Lambda$.  Then, as $R \rightarrow 1$, $\nabla \psi(R\,\mathbb{B}_d)$ tends to ${\mathcal X}$ in Hausdorff distance:
$$\lim_{R \rightarrow 1}\mathrm{d}_H\big(\nabla \psi(R\,\mathbb{B}_d),{\mathcal X}\big) =0.
%\nabla \psi(R\,\mathbb{B}_d) = {\mathcal X}
$$
\end{Lemma}
\begin{proof}
Since $\nabla \psi(R\,\mathbb{B}_d)$ is contained in ${\mathcal X}$, we only need to analyse one of the two members of the maximum  defining  the Hausdorff distance: indeed, 
\begin{align*}
	\mathrm{d}_H(\nabla \psi(R\,\mathbb{B}_d),{\mathcal X})&= \max\{\sup_{{\bf a}\in \nabla \psi(R\,\mathbb{B}_d)}\inf_{_{\scriptstyle{\bf x}\in {\mathcal X}}}\vert {\bf a}-{\bf x}\vert,\sup_{{\bf x}\in {\mathcal X}} \inf_{_{\scriptstyle{\bf a}\in \nabla \psi(R\,\mathbb{B}_d)}}\vert {\bf a}-{\bf x}\vert \}\\
	&= \sup_{{\bf y}\in \mathbb{B}_d}\inf_{_{\scriptstyle{\bf b}\in R\,\mathbb{B}_d}}\vert \nabla \psi({\bf b})-\nabla \psi({\bf y})\vert  .
\end{align*}
On the other hand, since $r\,\mathbb{B}_d\subset R\,\mathbb{B}_d\subset {\mathcal X}$,  $\nabla \psi(r\,\mathbb{B}_d)\subset \nabla \psi(R\,\mathbb{B}_d)\subset {\mathcal X}$ for $r\leq R$, so that the mapping $R\mapsto\mathrm{d}_H(\nabla \psi(R\,\mathbb{B}_d),{\mathcal X})$ is a decreasing function. Suppose that $\mathrm{d}_H(\nabla \psi(R\,\mathbb{B}_d),{\mathcal X})$ does not tend to $0$ when $R$ tends to $1$. Then, there exists $\epsilon>0$ such that, for every $R$, $\mathrm{d}_H(\nabla \psi(R\,\mathbb{B}_d),{\mathcal X})>\epsilon$;  in particular, there exists ${\bf x}_R\in {\mathcal X}$ such that $\vert {\bf a}_R-{\bf x}_R\vert>\epsilon$ %---equivalently,  for every $R$,    $\mathrm{d}(\nabla \psi(R\,\mathbb{B}_d),{\bf x}_R)>\epsilon$
 for all ${\bf a}_R\in  \nabla \psi(R\,\mathbb{B}_d)$. 

Now,  for each $n\in \N $, consider the sequences $A_n:=\nabla\psi\big((1-1/n)\mathbb{B}_d\big)$ and %the sequence of points
  $ {\bf y}_{n}:={{\bf x}_{1-1/n}}\in {\mathcal X}$. % previously described. 
  These sequences are such that  
  $$\inf_{{\bf a}\in A_m}\big\vert {\bf a}-{\bf y}_n\big\vert\geq \inf_{{\bf a}\in A_n} \big\vert  {\bf a}-{\bf y}_n\big\vert >\epsilon\quad\text{ for all $m\leq n$}.$$
By compactness, the sequence ${\bf y}_n$ admits a convergent subsequence, with limit  ${\bf y}_\infty$, say, with ${\bf y}_\infty \in {\mathcal X}$. This limit   satisfies $\inf_{{\bf a}\in A_n} \big\vert  {\bf a}- {\bf y}_\infty\big\vert >\epsilon$ for all $n\in \N$, which is not possible since  ${\mathcal X}=\bigcup_{n\in \N}A_{n}$.
\end{proof}

Our final result concerns the shape of the quantile contours of smooth probability measures (those satisfying the assumptions of Theorem~\ref{th:homeomorphism}).
As a consequence of Theorem~\ref{th:homeomorphism}, the sets $\mathbf{Q}_{\pm}(r\,\mathbb{B}_d)$ are bounded,  with
 connected boundary. Beyond this type of topological properties, results on the geometry of the quantile regions are not available. 
Here we prove that they satisfy a weak form of convexity. Recall from \cite{Cholaquidadisetal} that a set $B\subseteq \mathbb{R}^d$ is $\rho$-lighthouse convex if, from every point $\mathbf{x}$ in the boundary of $B$, there exists an open cone with vertex   $\mathbf{x}$ and opening angle $\rho>0$ which is contained in $\mathbb{R}^d\setminus B$. The limiting version of this concept (obtained as~$\rho\to 0$) is that for every point $\mathbf{x}$ in the boundary of $B$ there exists 
a ray emanating from $\mathbf{x}$ that does not intersect $B$ at any other point. This is precisely what can be proved for quantile sets.
\begin{Lemma}\label{lem:FORM}
Let $\rm P$ be a probability measure   on $\mathbb{R}^d$ satisfying the assumptions of Theorem~\ref{th:homeomorphism}.  Then, for all $r\in(0,1)$ and all $\bf y$ belonging to the boundary of $ \mathbf{Q}_{\pm}(r\,\mathbb{B}_d)$, there exists a ray $T$ emanating from $\mathbf{y}$ for which $\overline{ \mathbf{Q}_{\pm}(r\,\mathbb{B}_d)}\cap T=\{ \bf y \}$. 
\end{Lemma}
\begin{proof}
Assume, on the contrary, that there exists $\bf y$ in the boundary of $\mathbf{Q}_{\pm}(r\,\mathbb{B}_d)$  such that for every ray  $T=\{{\bf z}\in \R^d : \ {\bf z} ={\bf y }+t {\bf s}, \ t\geq 0\}$, 
there exists in $\overline{ \mathbf{Q}_{\pm}(r\,\mathbb{B}_d)}\cap T$ at least one point $\bf z$ distinct from~$\bf y$. Note that, necessarily, that point can be chosen in the boundary of $\mathbf{Q}_{\pm}(r\,\mathbb{B}_d)$. Now, since $\textbf{Q}_{\pm}$ is a homeomorphism, it maps boundaries into boundaries. Therefore, we can assume, up to a rotation, that ${\bf y}={\bf Q}_{\pm}( r {\bf e}_1)$. Monotonicity of $\textbf{Q}_{\pm}$ implies that 
\begin{align}\label{eq:mono_cond}
\langle \textbf{u} -r\textbf{e}_1 ,\textbf{Q}_{\pm}(\textbf{u})-\textbf{Q}_{\pm}(r\textbf{e}_1)\rangle \geq 0\qquad \text{for all $\textbf{u} \in \mathbb{B}_d$}.
\end{align}
However, if $\mathbf{Q}_{\pm}(\textbf{u})\in T=\{{\bf z}\in \R^d : \ {\bf z} ={\bf y }+t {\bf e}_1, \ t\geq 0\}$, then $\textbf{Q}_{\pm}(\textbf{u})= \textbf{Q}_{\pm}(r{\bf e}_1)+t {\bf e}_1$ for some~$t>0$. Hence, by \eqref{eq:mono_cond}  $\langle \textbf{u} -r \textbf{e}_1  , r {\bf e}_1 \rangle \geq 0  $. This implies that $\mathbf{u}\notin r\,\mathbb{B}_d$, thus contradicting the assumption that $T$ has a common point with the boundary of $\mathbf{Q}_{\pm}(r\,\mathbb{B}_d)$ other than $\bf y$. 
\end{proof}

\end{document}